\documentclass[a4paper, 11point]{amsart}

\newcommand{\RR}{\mathbb{R}} 
\newcommand{\CC}{\mathbb{C}} 

\newcommand{\RE}{\mathrm{Re}\,}
\newcommand{\IM}{\mathrm{Im}\,}
\newcommand{\Aut}{\mathrm{Aut}}

\newcommand{\wt}{\widetilde}

\newcommand{\w}{\wedge}
\newcommand{\beq}{\begin{eqnarray*}}
\newcommand{\eeq}{\end{eqnarray*}}
\newcommand{\beg}{\begin{equation}}
\newcommand{\eeg}{\end{equation}}
\newcommand{\om}{\omega}
\newcommand{\Om}{\Omega}

\newcommand{\la}{\lambda}

\newcommand{\al}{\alpha}
\newcommand{\be}{\beta}
\newcommand{\ga}{\gamma}

\newcommand{\ov}{\overline}

\newtheorem{thm}{Theorem}[section]

\newtheorem{lem}[thm]{Lemma}
\newtheorem{prop}[thm]{Proposition}
\theoremstyle{definition}

\theoremstyle{remark}

\numberwithin{equation}{section}

\title[Semicontinuity property for automorphism groups]{A generalization of the Greene-Krantz theorem for the semicontinuity property of automorphism groups}

\author{Jae-Cheon Joo}

\address{Department of Mathematics and Informatics, University of Wuppertal, Gaussstr. 20, D-42119 Wuppertal, Germany}

\email{jcjoo91@uni-wuppertal.de}

\begin{document}

\begin{abstract}
A CR version of the Greene-Krantz theorem \cite{GK} for the semicontinuity of complex automorphism groups will be provided. This is not only a generalization but also an intrinsic interpretation of the Greene-Krantz theorem. 
\end{abstract}

\maketitle

\section{Introduction}
By the uppersemicontinuity, or simply, the semicontinuity in geometry, we mean the property that the set of symmetries of a geometric structure should not decrease at a limit of a sequence of the structures. For instance, a sequence of ellipses in the Euclidean plane can converge to a circle, while a sequence of circles cannot converge to non-circular ellipse. This property seems as natural as the second law of thermodynamics in physics, but we still need to make it clear in mathematical terminology. A symmetry for a geometric structure is described as a transformation on a space with the geometric structure. The set of transformations becomes a group with respect to the composition operator. Therefore, the semicontinuity can be understood as a non-decreasing property of the transformation groups at the limit of a sequence of geometric structures. One of the strongest version describing the semicontinuity was obtained by Ebin for the Riemannian structures on compact manifolds in terms of conjugations by diffeomorphisms.

\begin{thm}[\cite{Eb}]\label{t;Ebin}
Let $M$ be a $C^\infty$-smooth compact manifold and let $\{g_j: j=1,2,...\}$ be a sequence of $C^\infty$-smooth Riemannian structures which converges to a Riemannian metric $g_0$ in $C^\infty$-sense. Then for each sufficiently large $j$, there exists a diffeomorphism $\phi_j:M\rightarrow M$ suh that $\phi_j\circ I_j \circ \phi_j^{-1}$ is a Lie subgroup of $I_0$, where $I_j$ and $I_0$ represent the isometry groups for $g_j$ and $g_0$, respectively.
\end{thm}

The group of holomorphic automorphisms on a complex manifold plays a role of the group of symmetry with respect to the complex structure. By Cartan's theorem (cf. \cite{GKK}), the automorphism group of a bounded domain in the complex Euclidean space turns out to be a Lie group with the compat-open topology on the domain. Greene and Krantz proved the following theorem for the semicontinuity property of automorphism groups of bounded strongly pseudoconvex domains.

\begin{thm}[\cite{GK}]\label{t;GK}
Let $\Om_j$ $(j=1,2,...)$ and $\Om_0$ be bounded strongly pseudoconvex domains in $\CC^n$ with $C^\infty$-smooth boundary. Suppose that $\Om_j$ converges to $\Om_0$ in $C^\infty$-sense, that is, there exists a diffeomorphism $\psi_j$ defined on a neighborhood of $\ov\Om_0$ into $\CC^n$ such that $\psi_j(\Om_0) = \Om_j$ and $\psi_j\rightarrow Id$ in $C^\infty$-sense on $\ov\Om_0$. Then for every sufficiently large $j$, there exists a diffeomorphism $\phi_j : \Om_j\rightarrow \Om_0$ such that $\phi_j\circ \Aut (\Om_j) \circ\phi_j^{-1}$ is a Lie subgroup of $\Aut(\Om_0)$.
\end{thm}

Unlike the isometry group of a compact Riemannian manifold, the holomorphic automorphism group on a bounded strongly pseudoconvex domain can be noncompact, so the proof of Theorem \ref{t;GK} is divided into two cases; Either $\Aut(\Om_0)$ is compact or not. It turned out that the latter case is relatively simple, which is the case of deformations of the unit ball by the Wong-Rosay theorem \cite{Ro, Wo}. The main part of the proof of Theoem \ref{t;GK} is thus devoted to the case when $\Aut(\Om_0)$ is compact.  Greene and Krantz proved this case by constructing a compact Riemannian manifold $(M, g_j)$ which includes $\Om_j$ as a relatively compact subset and whose isometry group contains the automorphism group of $\Om_j$. Then Ebin's theorem yields the conclusion. The Riemannian manifold $(M, g_j)$ is called a {\em metric double} of $\Om_j$.

This idea of proof is applicable to more general cases. One of reasonable ways of generalization is  to prove the semicontinuity property for more general class of domains. In recent paper \cite{GKKS}, the authors generalized Theorem \ref{t;GK} to finitely differentiable cases. In \cite{GKi}, Greene and Kim proved that a partial generalization is also possible even for some classes of non-strongly pseudoconvex domains. 

\medskip

The aim of the present paper is to obtain another generalization of Theorem \ref{t;GK}. According to Hamilton's theorem in \cite{Ha1, Ha2}, defomations of a bounded strongly pseudoconvex domain with $C^\infty$-smooth boundary coincide with deformations of complex structure on a given domain and they give rise to deformations of CR structure of the boundary.  Fefferman's extension theorem \cite{Fe} shows that every holomorphic automorphism on a bounded strongly pseudoconvex domain with $C^\infty$-smooth boundary extends to a diffeomorphism up to the boundary and hence gives rise to a CR automorphism on the boundary. Conversely, a CR automorphism on the boundary extends a holomorphic automorphism on the domain by the Bochner-Hartogs extension theorem. It is also known that the compact-open topology of the automorphism group of the domain coincides with the $C^\infty$-topology of the CR automorphism group of the boundary (cf. \cite{Be}), if the holomorphic automorphism group of the domain is compact. In this observation, it is natural to think of the semicontinuity property for abstract strongly pseudoconevex CR manifolds under deformations of CR structures, as a generalization of Theorem \ref{t;GK}. We prove the following theorem for CR automorphism groups when the limit structure has a compact CR automorphism group.

\begin{thm}\label{t;main}
Let $\{J_k : k=1, 2, \ldots\}$ be a sequence of $C^\infty$-smooth strongly pseudoconvex CR structures on a compact differentiable manifold $M$ of dimension $2n+1$, which converges to a $C^\infty$-smooth strongly pseudoconvex CR structure $J_0$ on $M$ in $C^\infty$-sense. Suppose that  the CR automorphism group $\Aut_{CR}(M, J_0)$ is compact. Then there exists $N>0$ and a diffeomorphism $\varphi_k:M\rightarrow M$ for each $k >N$ such that $\varphi_k \circ \Aut_{CR}(M, J_k) \circ \varphi^{-1}_k$ is a Lie subgroup of $\Aut_{CR}(M, J_0)$. 
\end{thm}

According to Schoen's theorem \cite{Sc}.  $\Aut_{CR}(M, J_0)$ is compact if and only if $(M, J_0)$ is not CR euivalent with the sphere $S^{2n+1}$ with the standard CR structure.  One should notice that this condition is not necessary if $2n+1 \geq 5$. Boutet de Monvel showed in \cite{Bo} that a CR structure on $M$ which is sufficiently close to the standard structure on $S^{2n+1}$ is also embeddable in $\CC^{n+1}$, if $2n+1 \geq 5$ in contrast with $3$-dimensional case (see \cite{BE, Lem, Ni, Ro}). Therefore, if $\Aut_{CR}(M, J_0)$ is noncompact and $2n+1\geq5$, then the situation is reduced to the case of defomations of the unit ball and follows immediately from Theorem \ref{t;GK}. 

The rest of this paper will be devoted to proving Theorem \ref{t;main}. Since we are thinking about abstract CR manifolds, it is needed to develop an intrinsic way of proof of this. Therefore, the main interest of Theorem \ref{t;main} is not only in the generalization  but also in the intrinsic verification of the Greene-Krantz theorem. The main tool of the proof is the solution for the CR Yamabe problem about the construction of pseudohermitian structures with constant Webster scalar curvature, which is intensively studied for instance, in \cite{CCP, Ga, GY, JL1, JL2}. The subellipticity of the CR Yamabe equation turned out quite useful in obtaining estimates of derivatives of CR automorphisms by Schoen in \cite{Sc}.  We make use of various solutions for the CR Yamabe problem - minimal solutions,  local scalar flattening solutions and the blowing-up solutions given by the Green funtions - developed in \cite{FiS, JL1, JL2, Sc} for the proof. 


\section{Strongly pseudoconvex CR manifolds} 
In this section, we summarize fundamental facts on the strongly pseudoconvex CR manifolds and pseudohermitian structures. The summation convention is always assumed.


\subsection{CR and pseudohermitian structures} Let $M$ be a smooth manifold of dimension $2n+1$ for some positive integer $n$. A {\em CR structure} on $M$ is a smooth complex structure $J$  on a subbundle $H$ of the rank $2n$ of the tangent bundle $TM$ which satisfies the integrability condition. More precisely, the restriction of $J$ on a fibre $H_p$ for a point $p\in M$ is an endomorphism $J_p : H_p\rightarrow H_p$ which satisfies $J_p\circ J_p = -Id_{H_p}$, varying smoothly as $p$ varies, and the bundle of $i$-eigenspace $H^{1,0}$ of $J$ in the complexification $\mathbb{C}\otimes H$ satisfies the Frobenius integrability condition 
$$[\Gamma(H^{1,0}), \Gamma(H^{1,0})]\subset \Gamma(H^{1,0}).$$ 
The subbundle $H$ is called the {\em CR distribution} of $J$.  A {\em CR automorphism }on $M$ is a smooth diffeomorphism $F$ from $M$ onto itself such that $F_* H^{1,0} = H^{1,0}$. We denote by $\Aut_{CR}(M)$ the group of all CR automorphisms on $M$. 
  A CR strucutre is said to be {\em strongly pseudoconvex} if its CR distribution $H$ is a contact distribution and for a contact form $\theta$, the {\em Levi form} $\mathcal{L}_\theta$ defined by 
$$\mathcal{L}_\theta (Z, \ov W):= -i\,d\theta (Z, \ov W)$$ 
for $Z, W\in H^{1,0}$ is positive definite. It is known that $C^0$-topology of $\Aut_{CR}(M)$ coincides with $C^\infty$-topology for a compact strongly pseudoconvex CR manifold $M$, if $\Aut_{CR}(M)$ is compact with respect to $C^0$-topology. See \cite{Sc} for the proof. 

We call a fixed contact form for the CR distribution of a strongly pseudoconvex CR structure a {\em pseudohermitian structure}. Let $\{W_\al : \al=1,...,n\}$ be a local frame, that is, $W_\al$'s are sections of $H^{1,0}$ which form a pointwise basis for $H_{1,0}$. We call a collection of $1$-forms $\{\theta^\al\}$ the {\em admissible coframe} of $\{W_\al\}$, if they are sections of $(H^{1,0})^*$ and satisfy 
$$\theta^\al (W_\be) = \delta^\al_\be,\quad \theta^\al(T) =0,$$
where $T$ is the vector field uniquely determined by 
$$\theta(T)=1,\quad T\lrcorner d\theta =0,$$
which is called the {\em characteristic vector field} for $\theta$.
let $g_{\al\bar\be} = \mathcal{L}_\theta (W_\al, W_{\bar\be})$. Then 
$$d\theta = 2i g_{\al\bar\be}\,\theta^\al\w\theta^{\bar\be},$$
where $\{\theta^\al\}$ is the admissible coframe for $\{W_\al\}$. 

\begin{thm}[\cite{We}]\label{t;webster} There exist a local $1$-form $\om = ({\om_\be}^\al)$ and local functions ${A^\al}_\be$ uniquely determined by 
$$d\theta^\al = \theta^\be\w{\om_\be}^\al + {A^\al}_{\bar\be}\,\theta\w\theta^{\bar\be},$$
$$dg_{\al\bar\be} = \om_{\al\bar\be} + \om_{\bar\be\al},\quad A_{\al\be}= A_{\be\al}.$$
\end{thm}

Here and the sequel, we lower or raise index by $(g_{\al\bar\be})$ and $(g^{\al\bar\be}) = (g_{\al\bar\be})^{-1}$. A connection $\nabla$ defined by 
$$\nabla W_\al = {\om_\al}^\be\otimes W_\be,\quad \nabla T =0$$
is called the {\em pseudohermitian connection} or the {\em Webster connection} for $\theta$. The functions ${A^\al}_\be$ are called the coefficients of the {\em torsion tensor} $\mathbf{T}$. Let 
$$d{\om_\al}^\be - {\om_\al}^\ga \w {\om_\ga}^\be\equiv {{R_\al}^\be}_{\ga\bar\sigma}\,\theta^\ga\w\theta^{\bar\sigma}\quad \mbox{mod } \theta, \theta^\ga\w\theta^\sigma, \theta^{\bar\ga}\w\theta^{\bar\sigma}.$$
We call ${{R_\al}^\be}_{\ga\bar\sigma}$ the coefficients of the {\em Webster curvature tensor} $\mathbf{R}$. Contracting indices, we obtain the coefficients $R_{\al\bar\be}$ of the Webster Ricci curvature $Ric$ and the Webster scalar curvature $S$ as follows:
$$R_{\al\bar\be} = {{R_\ga}^\ga}_{\al\bar\be},\quad S = R_{\al\bar\be} \,g^{\al\bar\be}.$$

The norm of the Webster curvature $|\mathbf R|_\theta$ is defined by 
$$|\mathbf R|^2_\theta = \sum_{\al, \be, \ga, \sigma} |{{R_\al}^\be}_{\ga\bar\sigma}|^2$$
where the frame is chosen so that $g_{\al\bar\be} = \delta_{\al\bar\be}$. We similarly define the norm of the torsion tensor $|\mathbf{T}|_\theta$.

A pseudohermitian strucutre defines a sub-Riemannian structure. The distance function induced by a sub-Riemannian metric is called the {\em Carnot-Carath\'eodory distance} (cf. \cite{St}). We denote by $B_\theta(x,r)$ the Carnot-Carath\'eodory ball with respect to the pseudohermitian structure $\theta$ of radius $r>0$ centered at $x\in M$.

\medskip

The {\em Heisenberg group} $\mathcal{H}^n$ is a strongly pseudoconvex CR manifold $\mathbb{C}^n\times \mathbb{R}$ with the CR structure whose $H^{1,0}$ bundle is spanned by 
\beg\label{e;st1}
Z_\al = \frac{\partial}{\partial z^\al} + i z^{\bar\al}\,\frac{\partial}{\partial t}, \quad \al =1,...,n,
\eeg
where $(z,t) = (z^1,...,z^n, t)$ is the standard coordinate system of $\CC^n\times\RR$. It is well-known that $\mathcal{H}^n$ is CR equivalent to the sphere in $\CC^{n+1}$ minus a single point. If we put 
\beg\label{e;st2}
\vartheta_0 = dt - iz^{\bar\al} \,dz^\al + iz^\al \,dz^{\bar\al},
\eeg
then it turns out the curvature and torsion tensors vanish identically. The converse also follows from the solution of the Cartan equivalence problem.
\begin{prop}\label{p;equi}
If the curvature and the torsion tensors of a pseudohermitian manifold $(M, \theta)$ vanish identically, then the pseudohermitian structure of $M$ is locally equivalent to that of $(\mathcal{H}^n, \vartheta_0)$. If we assume more that $M$ is simply connected and complete in the sense that every Carnot-Carath\'eodory ball is relatively compact in $M$, then $(M, \theta)$ is globally equivalent to $(\mathcal{H}^n, \vartheta_0)$,
\end{prop}

For a given pseudohermitian manifold $(M,\theta)$, we can extend the CR structure $J$ to a smooth section of endomorphism $\hat J$ on $TM$ by putting $\hat J(T) =0$, where $T$ is the characteristic vector field of $\theta$. Let $J_k$, $k=1,2,\ldots$ and $J_0$ be strongly pseudoconvex CR structures on $M$ with CR distributions $H_k$ and $H_0$, respectively. We say that $J_k$ conevrges to $J_0$ in $C^l$-sense ($l=0,1,2,...,\infty$), if there exists pseudohermitian structure $\theta_k$ and $\theta_0$ for $(M, J_k)$ and $(M, J_0)$ such that $\theta_k \rightarrow \theta_0$ and $\hat J_k \rightarrow \hat J_0$ in $C^l$-sense as tensors on $M$.


\subsection{Pseudoconformal change of structures and the CR Yamabe equation}
Let $(M, \theta)$ be a $(2n+1)$-dimensional pseudohermitian manifold and let $\tilde \theta = e^{2f} \theta$ be a pseudoconformal change, where $f$ is a smooth real-valued function. Let $\{\theta^\al\}$ be an admissible coframe for $\theta$ satisfying $d\theta = 2i g_{\al\bar\be}\,\theta^\al\w\theta^{\bar\be}$. Then it turns out 
$$\tilde\theta^\al = e^f (\theta^\al + if^\al \,\theta),\quad \al=1,...n$$
form an admissible coframe for $\tilde\theta$ which satisfies 
$$d\tilde\theta = 2i g_{\al\bar\be}\,\tilde\theta^\al\w\tilde\theta^{\bar\be}.$$
Let ${{R_\al}^\be}_{\ga\bar\sigma}$ and $\wt {{R_\al}^\be}_{\ga\bar\sigma}$ be coefficients of the Webster curvatures for $\theta$ and $\tilde\theta$ evaluated in the coframes $\{\theta^\al\}$ and $\{\tilde\theta^\al\}$, respectively. Then they are related as 
\begin{eqnarray}\label{e;trans_curv}
\wt {{R_\al}^\be}_{\ga\bar\sigma} &=& e^{-2f} \left\{ {{R_\al}^\be}_{\ga\bar\sigma} - {\delta_\al}^\be \left(f_{\ga\bar\sigma} + f_{\bar\sigma\ga}\right)  - 2g_{\al\bar\sigma}\,{f^\be}_\ga -2 f_{\al\bar\sigma}{\delta^\be}_\ga \right.\\
& &\left. - \left({f^\be}_\al + {f_\al}^\be\right)g_{\ga\bar\sigma} - 4\left({\delta_\al}^\be\, g_{\ga\bar\sigma} + g_{\al\bar\sigma}\,{\delta^\be}_\ga \right)f^\la f_\la \right\}\nonumber
\end{eqnarray}
where $f_{\al\bar\be}$, ${f_\al}^\be$ and ${f^\be}_\al$ are components of the second covariant derivatives of $f$ of the pseudohermitian manifold $(M,\theta)$ (cf. Proposition 4.14 in \cite{Jo-Le} for more general case). Contracting indices, we obtain the following transformation formula for Webster the scalar curvatures:
\beg\label{e;trans_scal}
\wt S = e^{-2f} \left\{S + 2(n+1) \Delta_\theta f - 4n(n+1) f^\la f_\la \right\}
\eeg
where $\Delta_\theta f = -\left({f_\al}^\al + {f_{\bar\al}}^{\bar\al} \right)$. The operator $\Delta_\theta$ is called the {\em sublaplacian} for $\theta$.

Let $u$ be a positive smooth function on $M$ defined by $u^{p-2} = e^{2f}$, where $p=2+2/n$. Then \eqref{e;trans_scal} changes into the following nonlinear equation for $u$:
\beg\label{e;yamabe}
L_\theta u: = \left(b_n \Delta_\theta  + S \right)u = \wt S\, u^{p-1},
\eeg
where $b_n = 2+2/n$ (see \cite{JL1, JL2, Le}). Equation \eqref{e;yamabe} is called the {\em CR Yamabe equation} and the subelliptic linear operator $L_\theta$ is called the {\em CR Laplacan} for $\theta$. The {\em CR Yamabe problem} is to find a positive smooth function $u$ which makes $\wt S$ constant. 

Let ${A^\al}_{\bar\be}$ and $\wt{A^\al}_{\bar\be}$ be the coefficients of the torsion tensors for $\theta$ and $\tilde\theta$ in the coframes $\{\theta^\al\}$ and $\{\tilde\theta^\al\}$, respectively. Then in turns out that 
\beg\label{e;trans_tor}
\wt{A^\al}_{\bar\be} = e^{-2f} \left({A^\al}_{\bar\be} - i{f^\al}_{\bar\be} + 2i f^\al f_{\bar\be}  \right).
\eeg 
See \cite{Le}. for details.


\subsection{Folland-Stein spaces and subelliptic estimates} Roughly speaking, a {\em normal coordinate system} of a pseudohermitian manifold $(M, \theta)$ of dimension $(2n+1)$  is a local appoximation by the standard pseudohermitian structure on the Heisenberg group $(\mathcal{H}^n, \theta_0)$. For $p\in M$, let $W_1,...,W_n$ be a local frame defined on a neighborhood $V$ of $p$ such that the coefficients of the Levi form for $\theta$ is $g_{\al\bar\be} = \delta_{\al\bar\be}$. Such a frame is called a {\em unitary frame}. We denote by $T$ the characteristic vector field for $\theta$. Let $(z, t)$ be the standard coordinates of $\mathcal{H}^n$ and let $|(z,t)| = (|z|^4 + t^2)^{1/4}$ the Heisenberg group norm. $Z_\al$ and $\theta_0$ are defined on $\mathcal{H}^n$ as \eqref{e;st1} and \eqref{e;st2}. Then 

\begin{thm}[\cite{FS}]\label{t;FS}
There is a neighborhood of the diagonal $\Omega \subset V \times V$ and a $C^\infty$-smooth
mapping $\Theta : \Omega \rightarrow \mathcal{H}^n$ satisfying:

\begin{itemize}
\item[(a)] $\Theta (\xi, \eta) = -\Theta(\eta, \xi) = \Theta (\eta, \xi )^{-1}$ . (In particular, $\Theta(\xi, \xi) = 0$.)

\item[(b)] Let $\Theta_\xi(\eta) = \Theta (\xi, \eta)$. Then $\Theta_\xi$ is a diffeomorphism of a neighborhood $\Omega_\xi$ of $\xi$ onto a neighborhood of the origin in $\mathcal{H}^n$. Denote by $y = (z, t) = \Theta(\xi, \eta)$
the coordinates of $\mathcal{H}^n$. Denote by $O^k$ ($k=1,2,\ldots$) a $C^\infty$ function $f$ of $\xi$ and $y$
such that for each compact set $K \subset V$, there is a constant $C_K$, with $f(\xi,y) \leq C_K |y|^k$ (Heisenberg norm) for $\xi\in K$. Then we have the following approximation formula. 
$$(\Theta_\xi^{-1})^* \theta = \theta_0 + O^1 d t+ \sum_{\al=1}^n (O^2 dz^\al + O^2dz^{\bar\al}), \quad (\Theta_\xi^{-1})^* (\theta\w d\theta^n) = (1 + O^1) \theta_0\w d\theta_0^n,$$
$$\Theta_{\xi*} W_\al = Z_\al + O^1 \mathcal{E}(\partial_z) + O^2 \mathcal{E}(\partial_t), \quad \Theta_{\xi*} T = \partial/\partial t + O^1 \mathcal{E}(\partial_z, \partial_t),$$
$$\Theta_{\xi*} \Delta_\theta = \Delta_{\theta_0}+ \mathcal{E}(\partial_z) + O^1 \mathcal{E}(\partial_t, \partial_z^2) + O^2\mathcal{E}(\partial_z\partial_t) + O^3\mathcal{E}(\partial_t^2).$$
\end{itemize}
Here $O^k \mathcal{E}$ indicates an operator involving linear combinations of the indicated
derivatives with smooth coefficients in $O^k$, and we have used $\partial_z$ to denote any of the
derivatives $\partial/\partial z^\al$, $\partial/\partial z^{\bar\al}$.
\end{thm} 
The smooth map $\Theta_\xi$ is called the {\em Folland-Stein normal coordinates} centered at $\xi$ with respect to the frame $\{W_\al\}$. (This coordinate system depends on the choice of local unitary frame. Another construction of pseudohermitian normal coordinates which does not depend on local frames is given in \cite{JL2}.) Here and in the sequel, we use the term {\em frame constants} to mean bounds on
finitely many derivatives of the coefficients in the $O^k\mathcal{E}$ terms in Theorem \ref{t;FS}.

Let $V$ be an open neighborhood of a point $p\in M$ with a fixed local unitary frame $W_1,...,W_n$ and let $U$ be a relatively compact open neighborhood of $p$ in $V$ such that $\Omega_\xi$ in Theorem \ref{t;FS} contains $\ov U$ for every $\xi\in \ov U$.  Let $X_\al = \RE W_\al$ and $X_{\al+n} = \IM W_\al$ for $\al=1,...,n$. For a multiindex $A = (\al_1,...,\al_k)$, with $1\leq \al_j\leq 2n$, $j=1,...,k$, we denote $k$ by $\ell(A)$ and denote
$X^A f = X_{\al_1}\cdots X_{\al_k} f$ for a smooth function $f$ on $U$. The $S^p_k(U)$-norm of a smooth function $f$ on $U$ is 
$$\|f\|_{S^p_k(U)} = \sup_{\ell(A)\leq k} \|X^Af\|_{L^p(U)},$$ 
where $\|g\|_{L^p(U)} = \left(\int_U |g|^p\,\theta\w d\theta^n \right)^{1/p}$ is the $L^p$-norm of $g$ on $U$ with respect to the volume element induced by $\theta$. The completion of $C^\infty_0(U)$ with respect to $\|\cdot\|_{S^p_k(U)}$ is denoted by $S^p_k(U)$.

H\"older type spaces suited to $\Delta_\theta$ is also defined as follows. For $x, y\in U$, let $\rho(x,y) = |\Theta(x,y)|$ (Heisenberg norm). For a positive real number $0<s<1$, 
$$\Gamma_s (U) = \{f\in C^0 (\ov U) : |f(x)-f(y)| \leq C \rho(x,y)^s \mbox{ for some constant }C>0\}.$$
If $s$ is a positive non-integral real number such that $k<s<k+1$ for some integer $k\geq 1$, then 
$$\Gamma_s(U) = \{f\in C^0(\ov U) : X^A f \in \Gamma_{s-k} (U), \,\, \ell(A)\leq k\}.$$
Then $\Gamma_s(U)$-norm for $f\in \Gamma_s(U)$ is defined by 
$$\|f\|_{\Gamma_s(U)} = \sup_{x\in U} |f(x)| + \sup\left\{\frac{|X^A f(x)-X^A f(y)|}{\rho(x,y)^{s-k}} : x, y\in U,\,\,x\neq y,\,\, \ell(A) \leq k\right\}.$$
The function spaces $S^p_k(U)$ and $\Gamma_s(U)$ are called the {\em Folland-Stein spaces} on $U$. We denote by $\Lambda_s(U)$ the Euclidean H\'older space when we regard $U$ as a subset of $\RR^{2n+1}$.

\begin{thm}[\cite{FS}]\label{t;subelliptic} For each positive real number $s$ which is not an integer, each $1<r<\infty$ and each integer $k\geq 1$, there exists a constant $C>$ such that for every $f\in C^\infty_0(U)$, 
\begin{itemize}
\item[(a)] $\|f\|_{\Gamma_s(U)} \leq C \|f\|_{S^r_k(U)}$, where $1/r = (k-s)/(2n+2)$,
\item[(b)] $\|f\|_{\Lambda_{s/2}(U)} \leq C \|f\|_{\Gamma_s(U)}$,
\item[(c)] $\|f\|_{S^r_2(U)} \leq C (\|\Delta_\theta f\|_{L^r(U)} + \|f\|_{L^r(U)})$,
\item[(d)] $\|f\|_{\Gamma_{s+2}(U)} \leq C (\|\Delta_\theta f\|_{\Gamma_s(U)} + \|f\|_{\Gamma_s(U)})$.
\end{itemize} 
Moreover the constant $C$ depends only on frame constants.
\end{thm}

One should notice that the constants $C$ in the theorem above depend on frame constants rarther than the pseudohermitian structure itself. Therefore, if $\mathcal U$ is a small neighborhood $(J_0, \theta_0)$ in $C^\infty$-topology, then we can choose constants $C$ in Theorem \ref{t;subelliptic} which are independent of choice $(J, \theta) \in \mathcal U$. This fact will be same for all theorems in this paper about some norm estimates with constants depending on frame constants.

If $M$ is compact, we can choose a finite open covering $U_1,...,U_m$ each of which is contained in a normal coordinates. Let $\phi_1,...,\phi_m$ be a partition of unity subordinate to this covering. Then the spaces of $S^p_k(M)$ and $\Gamma_s(M)$ are defined as spaces of function $u$ such that $\phi_ju \in S^p_k(U_j)$ or $\phi_j u \in \Gamma_s(U_j)$, respectively, for every $j=1,...,m$.


\section{Proof of Theorem \ref{t;main}}

The proof of Theorem \ref{t;main} is based on the following fundamental fact about the semicontinuity property of compact group actions proved by Ebin in \cite{Eb} for Theorem \ref{t;Ebin}. We denote by $\mathrm{Diff}(M)$ the group of $C^\infty$-smooth diffeomorphisms. Recall that the $C^\infty$-topology on $\mathrm{Diff}(M)$ is metrizable. We denote a metric inducing the $C^\infty$-topology by $d$.

\begin{thm}[\cite{Eb}; cf. \cite{GKK, GKKS, GrKa, Kim}]\label{t;gp}
Let $M$ be a compact $C^\infty$-smooth manifold and let $G_k$ $(k=1,2,\ldots)$ and $G_0$ be compact subgroups of $\mathrm{Diff}(M)$. Suppose $G_j\rightarrow G_0$ in $C^\infty$-topology as $j\rightarrow \infty$, that is, for every $\epsilon>0$, there exists an integer $N$ such that $d(f, G_0):=\inf_{g\in G_0} d(f,g) <\epsilon$ for every $f\in G_j$, whenever $j>N$. Then $G_j$ is isomorphic to a subgroup of $G_0$ for every sufficiently large $j$. Moreover, the isomorphism can be obtained by the conjugation by a diffeomorphism $\varphi_j$ of $M$ which converges to the identity map in $C^\infty$-sense.
\end{thm}
Therefore, it suffices to prove the following proposition for the conclusion of Theorem \ref{t;main}.

\begin{prop}\label{p;compact}
Let $\{J_k: k=1.2.\ldots\}$ be a sequence of strongly pseudoconvex CR structures on a compact manifold $M$ which tends to a strongly pseudoconvex CR structure $J_0$ as in Theorem \ref{t;main}. Suppose that $\Aut_{CR}(M, J_0)$ is compact.  Then $\Aut_{CR}(M,J_k)$ is also compact for every sufficiently large $k$. Futhermore, every sequence $\{F_k\in \Aut_{CR}(M, J_k):k=1,2,\ldots\}$ admits a subsequence converging to an element $F\in \Aut_{CR}(M, J_0)$ in $C^\infty$-sense.
\end{prop}

We will make use of the solutions of the CR Yamabe problem for the proof of Proposition \ref{p;compact}. According to the variational approach introduced by Jerison-Lee in \cite{JL1, JL2}, it is very natural to consider the sign of CR Yamabe invariant defined as follows: Let $(M,\theta)$ be a compact pseudohermitian manifold. For a $C^\infty$-smooth real-valued function $u$,
$$A(\theta; u) :=\int_M u\,L_\theta u\,\theta\w d\theta^n =  \int_M (b_n |du|_\theta^2 + R\,u^2)\, \theta\w d\theta^n$$
and let 
$$B(\theta; u) := \int_M |u|^p\,\theta\w d\theta^n.$$ Then the {\em CR Yamabe invariant} $Y(M)$ is defined by 
$$Y(M) := \inf\{ A(\theta; u) : u\in C^\infty(M), \,\,B(\theta; u)=1\}.$$
It is well-known that $Y(M)$ does not depend on the choice of contact form $\theta$. Let $J_k$ be a sequence of strongly pseudoconvex CR structures on $M$ tending to a strongly pseudoconvex CR structure $J_0$ as $k\rightarrow \infty$. We denote by $Y_k$ the CR Yamabe invarint of $(M, J_k)$. For the proof, we may assume either that $Y_k \leq 0$ for every $k$ or that $Y_k >0$ for every $k$. 


\subsection{Case $Y_k \leq 0$.} In this case, we use the minimal solution of the Yamabe problem. 

\begin{thm}[\cite{JL1}] \label{t;jl}
Let $M$ be a compact strongly pseudoconvex CR manifold of dimension $2n+1$. Then 
\begin{itemize}
\item[(\romannumeral1)] $Y(S^{2n+1}) >0$, where $Y(S^{2n+1})$ is the CR Yamabe invariant for the sphere $S^{2n+1}$ with the standard structure.

\item[(\romannumeral2)] $Y(M) \leq Y (S^{2n+1})$, 

\item[(\romannumeral3)] If $Y(M) < Y(S^{2n+1})$, then there exists a positive $C^\infty$-smooth fuunction $u$ which satisfies that $B (\theta; u) =1$ and $A(\theta; u) = Y(M)$ for a given pseudohermitian structure $\theta$.
This function $u$ satisfies 
$$L_\theta u = Y (M) u^{p-1}.$$ That is, the pseudohermitian structure $\tilde \theta = u^{p-2} \theta$ has a constant Webster scalar curvature $\wt R = Y(M)$. 
\end{itemize}
\end{thm} 

It is known in \cite{JL2} that $Y(M) < Y(S^{2n+1})$ if $M$ is not locally spherical and $2n+1 \geq 5$. The cases that $2n+1=3$ or $M$ is spherical are dealt by Gamara and Gamara-Yacoub in \cite{Ga, GY}.

\begin{prop}[Theorem 7.1 in \cite{JL1}]\label{p;unique}
If $Y(M) \leq 0$, then a pseudohermitian structure with constant Webster scalar curvature is unique up to constant multiples. As a consequnce, there is a unique pseudohermitian structure with constant Webster scalar curvature under the unit volume condition, if $Y(M) \leq 0$.
\end{prop}

\begin{prop}[Theorem 5.15 in \cite{JL1}]\label{p;est}
Let $M$ be a compact strongly pseudoconvex CR manifold of dimension $2n+1$ and let $\theta$ be a pseudohermitian structure. Suppose that $f, g \in C^\infty (M)$, $u\geq 0$, $u\in L^r$ for some $r>p = 2+2/n$ and 
$$\Delta_\theta u + g u = f u^{q-1}$$ in the distribution sense for some $2\leq q\leq p$. Then $u\in C^\infty(M)$, $u>0$. Futhermore, $\|u\|_{C^k}$ depends only on $\|u\|_{L^r}, \|f\|_{C^k}, \|g\|_{C^k}$ and frame constants, but not on $q$.
\end{prop}

Indeed. the above lemma is stated in a local version in \cite{JL1}. But it is obvious it holds globally by taking a partition of unity subordinate to a chart of normal coordinates. 

\begin{prop}[Case $k=1$, $r=2$ and $s=p$ in Proposition 5.5 of \cite{JL1}]\label{p;sobolev}
For a compact pseudohermitian manifold $(M, \theta)$ of dimension $2n+1$, there exists a constant $C>0$ such that 
$$\int_M |v|^p \,\theta\w d\theta^n \leq C \int_M (|dv|_\theta^2 + |v|^2) \theta\w d\theta^n$$ 
for every $C^\infty$-smooth function $v$ on $M$. 
\end{prop}

Since we are considering CR structures converging to the target structure $J_0$, we can choose also a sequence $\{\theta_k\}$ of contact forms which tends to a target pseudohermitian structure $\theta_0$ in $C^\infty$-sense. Without loss of generality, we always assume that $\int_M \theta_k\w d\theta_k^n =1$ for every $k$.

\begin{lem}\label{l;nonpositive}
Suppose that $Y_k \leq 0$ for every $k$. Let $u_k>0$ be the (unique) solution as in {\em (\romannumeral3)} of Theorem \ref{t;jl} with respect to $(J_k, \theta_k)$. Then for each nonnegative integer $l$, there exists a constant $C$ such that $\|u_k\|_{C^l} \leq C$ for every $k$.
\end{lem}

\begin{proof}
Since $u_k$ satisfies 
\beg\label{e;eqn}
b_n \Delta_{\theta_k} u_k + R_k u_k = Y_k u_k^{p-1},
\eeg
where $R_k$ is the Webster scalar curvature for $\theta_k$, we have 
$$\int_M \frac{p-1}{2}\,b_n u_k^{p-2}|du_k|_{\theta_k} \theta_k\w d\theta_k^n \leq \int_M |R_k u_k^p| \theta_k\w d\theta_k^n$$
by integrating after multiplying $u_k^{p-1}$ to the both side of \eqref{e;eqn}, since $Y_k \leq 0$. Therefore, the function $w_k := u_k^{p/2}$ satisfies
$$\int_M |dw_k|_{\theta_k}^2 \,\theta_k\w d\theta_k^n \leq C \int_M w_k^2 \,\theta_k\w d\theta_k^n= C\int_M u_k^p\, \theta_k\w d\theta_k^n =C,$$ since $R_k$ is bounded uniformly for $k$. Moreover since $(J_k,\theta_k) \rightarrow (J_0, \theta_0)$ in $C^\infty$-sense, Proposition \ref{p;sobolev} implies that there exists a constant $C>0$ independent of $k$ such that 
$$\int_M w^p_k \,\theta_k\w d\theta_k^n \leq C\int_M (|dw_k|_{\theta_k}^2 + w_k^2)\,\theta_k\w d\theta_k^n,$$ which is uniformly bounded for every $k$. This implies that $\|u_k\|_{L^r}$ is uniformly bounded as $(J_k, \theta_k) \rightarrow (J_0, \theta_0)$, where $r=p^2/2 >p$. Then the conclusion follows from Proposition \ref{p;est}, since frame constants for $(J_k, \theta_k)$ are also uniformly bounded as $(J_k, \theta_k) \rightarrow (J_0, \theta_0)$ in $C^\infty$-sense.
\end{proof}

\medskip

If $Y_k \leq 0$ for every $k\geq1$, then by taking a subsequnce, we may assume the sequence $\{u_k\}$ of solutions of the Yamabe problem with respect to $(J_k, \theta_k)$ converges to $u_0$, the solution of the Yamabe problem with respect to $(J_0, \theta_0)$ in $C^\infty$-sense by Lemma \ref{l;nonpositive}. Replacing $\theta_k$ by $u_k^{p-2} \theta_k$, then we may assume the Webster scalar curvature of $\theta_k$ is a nonpositive constant for every $k$. In this case, it is known that the CR automorphism group of $(M, J_k)$ coincides with the pseudohermitian automorphism group for $(M, J_k, \theta_k)$. Let $g_k$ be the Riemannian metric on $M$ defined by 
$$g_k = \theta_k\otimes \theta_k + d\theta_k (\cdot, J_k\cdot)$$
for each $k$. Then we see that $g_k\rightarrow g_0 = \theta_0\otimes\theta_0 + d\theta_0 (\cdot, J_0\cdot)$ in $C^\infty$-sense, and the CR automorphism groups $\Aut_{CR}(M, J_k)$ and $\Aut_{CR}(M,J_0)$ are subgroups of the isometry groups of $g_k$ and $g_0$, respectively. Then the conclusion follows from the proof of Theorem \ref{t;Ebin}.

\subsection{Case $Y_k >0$.} We start from proving that it is enough to consider $C^0$-convergence instead of $C^\infty$-convergence.

\begin{lem}\label{l;conv}
Suppose for a sequence $\{F_k \in \Aut_{CR}(M, J_k)\}$, $F_k\rightarrow F$ and $F^{-1}_k \rightarrow G$ in $C^0$-sense for some continuous mappings $F$ and $G$. Then $F\in \Aut_{CR}(M, J_0)$, $G= F^{-1}$ and $F_k\rightarrow F$ in $C^\infty$-sense.
\end{lem}

\begin{proof}
This lemma is a sequential version of Proposition 1.1' in \cite{Sc}. Let $\theta_k$ and $\theta_0$ be pseudohermitian structures for $J_k$ and $J_0$, respectively and suppose $\theta_k\rightarrow \theta_0$ in $C^\infty$-sense. For a given point $p\in M$, let $q_k = F_k(p)$ and $q=F(p)$. Let $q\in \wt U \subset\subset \wt V\subset\subset\wt W$ be relatively compact neighborhoods of $q$.  Since $q_k\rightarrow q$, we can assume that $q_k \in \wt U$ for every $k$. The fact $Y_k >0$ implies that the principal eigenvalue of $L_{\theta_k}$ on $M$ and hence the Dirichlet principal eigenvalue of $L_{\theta_k}$ on $\wt W$ are also positive for every $k$. Then the local scalar flattening argument by Fischer-Colbrie and Schoen (see \cite{FiS, Sc}), we have a positive $C^\infty$-smooth function $u_k$ on $\wt W$ such that $L_{\theta_k} u_k =0$ on $\wt W$ for every $k$. Multiplying a positive constant, we may assume that $u_k (q)=1$ for every $k$. Then the subelliptic theory Theorem \ref{t;FS} for sublaplacian and the Harnack principle (cf. Proposition 5.12 in \cite{JL1}) imply that $\{u_k\}$ has a convergent subsequence which tends to a positive function $u_0$ on the closure of  $\wt V$ in $C^\infty$-sense. We denote the convergent subsequence by $\{u_k\}$ again. Then $\tilde\theta_k = u_k^{p-2}\theta_k$ and $\tilde\theta_0=u_0^{p-2}\theta_0$ have the trivial Webster scalar curvatures on $\wt V$. From the equicontinuity of the sequence $\{F_k\}$, we can choose a neighborhood $W$ of $p$ such that $F_k(W) \in \wt U$ for every $k$. Let $v_k$ be a positive smooth function on $V$ defined by $F_k^* \tilde\theta_k = v_k^{p-2} \theta_k$. Then for every $k$, we have 
\beg\label{e;local}
L_{\theta_k} v_k = 0\quad \mbox{ on }W.
\eeg
We denote by $\mathrm{Vol}_{\tilde\theta_k}(\wt U)$ be the volume of $\wt U$ with respect to the volume form $\tilde\theta_k\w d\tilde\theta_k^n$. Since $\tilde\theta_k \rightarrow \tilde\theta_0$ in $C^\infty$-sense in $\wt V$, there exists a uniform bound $C$ of $\mathrm{Vol}_{\tilde\theta_k}(\wt U)$. Therefore, it turns out that
$$\int_W v_k^p \,\theta_k\w d\theta_k^n  = \int_W F^*_k (\tilde\theta_k \w d\tilde\theta_k) = \mathrm{Vol}_{\tilde\theta_k}(F_k(W))\leq \mathrm{Vol}_{\tilde\theta_k}(\wt U) \leq C$$
for every $k$. Fix a neighborhood $V\subset\subset W$ of $p$. Then the subelliptic mean-value inequality for \eqref{e;local} implies that there exists a constant $C$ such that $v_k(x) \leq C$ for every $x\in V$. We can also choose this $C$ indepedently on $k$ by the convergence of structures. Then for a given neighborhood $U\subset\subset V$ of $p$ and for each positive integer $l$, there exists a constant $C_l$ which is independent of $k$ such that 
$$\|v_k\|_{C^l(U)} \leq C_l$$
for every $k$, by Theorem \ref{t;FS}. Since each $F_k$ is pseudoconformal, the $C^l$-norm of $F_k$ on $U$ is completely determined by that of $v_k$ and is uniformly bounded on $U$. This yields that every subsequence of $\{F_k\}$ contains a subsequence converging in $C^l$-sense, for every positive integer $l$. Since $F_k$ converges to $F$ in $C^0$-sense on $M$ and since $M$ is compact, we conclude that $F_k$ converges to $F$ in  $C^\infty$-sense. By the same reason, $F^{-1}_k \rightarrow G$ in $C^\infty$-sense. It follows immediately that $F\in \Aut_{CR}(M, J_0)$ and $G=F^{-1}$.
\end{proof}

For a CR diffeomorphism $F :(M, \theta)\rightarrow (\wt M, \tilde\theta)$ between two pseudohermitian manifolds, we denote by $|F'|_{\theta, \tilde\theta}$ the pseudoconformal factor of $F$, thaat is, $F^*\tilde\theta = |F'|_{\theta,\tilde\theta} \,\theta$. We abbreviate it to $|F'|_\theta$ in case $(M, \theta) = (\wt M, \tilde\theta)$. 

\begin{lem}\label{l;expand}
Let $(M, \theta)$ and $(\wt M, \tilde\theta)$ be pseudohermitian manifolds of same dimension. Let $K$ be a relatively compact subset of $M$ and suppose that the Webster scalar curvature for $\tilde\theta$ vanishes on $\wt M$. Then there exist constants $r_0>0$ and $C>0$ such that for every CR diffeomorphism $F$ on a Carnot-Carath\'eodory ball $B_\theta (x, r)$ into $M$, 
$$B_{\tilde\theta}(F(x), C^{-1}\la r) \subset F(B_\theta(x,r)) \subset B_{\tilde\theta}(F(x), C\la r)$$ 
whenever $x\in K$ and $r\leq r_0/2$, where $\la = |F'|_{\theta,\tilde\theta}(x)$. The constant $C$ depends only on $r_0$, $K$ and uniform bounds of finite order derivatives of the CR and pseudohermitian structures of $(M, \theta)$. 
\end{lem}

This lemma is a restatement of (\romannumeral1) of Proposition 2.1' in \cite{Sc}, which is a consequence of the subelliptic Harnack principle.

\medskip

Assume the contrary to the conclusion of Proposition \ref{p;compact} for the proof. Then  there exisrs a sequence $\{F_k\in \Aut_{CR}(M,J_k)\}$ such that $\sup_{x\in M} |F_k'|_{\theta_k}(x)\rightarrow \infty$ as $k\rightarrow \infty$,  thanks to Lemma \ref{l;conv}. Let $x_k\in M$ be a point of $M$ with $|F_k'|_{\theta_k}(x_k) = \sup_{x\in M} |F_k'|_{\theta_k}(x)$. Extracting a subsequence, we assume that $x_k \rightarrow x_0 \in M$ and $F_k(x_k)\rightarrow z_0$ as $k\rightarrow \infty$.  Choose $r>0$ small enough that the Carnot-Carath\'eodory balls satisfy $B_{\theta_k}(x_k, r) \subset\subset B_{\theta_k}(x_k, 2r) \subset\subset U$ for each $k$, where $U$ is a relatively compact neighborhood of $x_0$ in $M$, and $2r < r_0$ for $r_0$ given in Lemma \ref{l;expand}.

\medskip

\begin{lem}\label{l;y_0} There exists a subsequence $\{F_{k_j}: j=1,2,...\}$ of $\{F_k:k=1,2,..\}$ which admits a point $y_0\in M$ such that for every compact subset $K$ in $M\setminus\{y_0\}$, there exists $N>0$ such that $K\subset F_{k_j} (B_{\theta_{k_j}}(x_{k_j}, 2r))$ if $k_j>N$. Moreover, for the subsequence, one can choose the point $y_0$ independently of  $r>0$ as $r\rightarrow 0$.
\end{lem}

\begin{proof} Suppose that for every $r>0$, there exists no sequence $\{y_k \in M\setminus {F_k(B_{\theta_k}(x_k, 2r))}\}$ such that $d (y_k, F_k(x_k) )> \epsilon$ for any given $\epsilon >0$, where $d$ is the sub-Riemannian distance induced from $\theta_0$. Then it turns out every sequence $\{y_k \in M\setminus F_k(B_{\theta_k}(x_k, 2r))\}$ converges to $z_0$. In this case, we just need to put $y_0 = z_0$. 

Now supose that  for some $r>0$, there exists a sequence $\{y_k \in M\setminus {F_k(B_{\theta_k}(x_k, 2r))}\}$ such that $d (y_k, F_k(x_k) )> \epsilon$ for infinitely many $k$ for some $\epsilon >0$. Extracting a subsequence, we may assume that $y_k\rightarrow y_0\in M$ and $d (y_k, F_k(x_k) )> \epsilon$ for every $k$ so that the sequence $\{F_k(x_k)\}$ is relatively compact in $M\setminus\{y_0\}$. Let $G_k$ be the Green function for $L_{\theta_k}$ with pole at $y_k$. Existence of the Green function follows from the fact that $Y_k  >0$ (see for instance, \cite{CCP, Ga}). We normalize $G_k$ by the condition $\min_{M\setminus \{y_k\}} G_k =1$.  Since each $G_k >0$ and $L_{\theta_k} G_k =0$ on $M\setminus\{y_k\}$, we may assume $\{G_k : k=1,2,...\}$ converges to a positive function $G_0$ on $M\setminus\{y_0\}$ in local $C^\infty$-sense, by extracting a subsequence if necessary. Let $\tilde\theta_k = G_k^{p-2} \theta_k$. Then $\tilde\theta_k$ is a pseudohermitiaan structure on $M\setminus \{y_k\}$ which is Webster scalar flat. Therefore, if we denote $\la_k = |F_k'|_{\theta_k,\tilde\theta_k}(x_k)$, then Lemma \ref{l;expand} implies that there exists a constant $C$ independent of $k$ such that 
$$B_{\tilde\theta_k}(F_k(x_k), C^{-1}\la_k r) \subset F_k(B_{\theta_k}(x_k,r)) \subset B_{\tilde\theta_k}(F_k(x_k), C\la_k r)$$ 
Since $G_k \geq 1$ and $|F_k'|_{\theta_k} (x_k) \rightarrow \infty$, $\la_k$ also tends to infinity  as $k\rightarrow\infty$. Therefore, a relatively compact subset $K$ in $M\setminus\{y_0\}$ should be included in $F_k(B_{\theta_k}(x_k,r))$ for every sufficiently large $k$, since $F_k(x_k)$ lies on a fixed relatively compact subset of $M\setminus\{y_0\}$ and $\tilde\theta_k \rightarrow \tilde\theta_0 =G_0^{p-2}\theta_0$ in local $C^\infty$-smooth sense on $M\setminus\{y_0\}$. Note that the choice of the sequence $\{y_k\}$ and $y_0$ still works for every $r' \leq r$. This yields the independence of $y_0$ on $r$ as $r\rightarrow 0$.
\end{proof}

As a consequence of Lemma \ref{l;y_0}, it turns out that $M\setminus\{y_0\}$ is simply connected and complete with respect to the sub-Riemannian distance induced by $\tilde\theta_0$. 

\medskip

Extracting a subsequence, we assume Lemma \ref{l;y_0} holds for the entire sequence $\{F_k\}$. Choose $y_k \in M\setminus F_k (B_{\theta_k} (x_k, 2r))$ which tends to $y_0$. Let $v_k$ and $f_k$ be real-valued functions on $B_{\theta_k}(x_k, 2r)$ defined by 
$$v_k^{p-2} = |F_k'|_{\theta_k, \tilde\theta_k} = e^{2f_k},$$
where $G_k$ is the normalized Green function for $L_{\theta_k}$ with pole at $y_k$ which converges to a positive function $G_0$ in local $C^\infty$-smooth sense on $M\setminus\{y_0\}$ as $k\rightarrow \infty$, and $\tilde\theta_k = G_k^{p-2} \theta_k$. 
Since $L_{\theta_k} v_k =0$, we see that there exists a constsnt $C$ independent of $k$ such that $ |F_k'|_{\theta_k, \tilde\theta_k} \geq C \la_k$ on $B_{\theta_k} (x_k, r)$ by the Harnack principle, where $\la_k = |F_k'|_{\theta_k, \tilde\theta_k}(x_k)$. Let $\{Z_k \in \Gamma(H^{1,0}_k)\}$ be a sequence of vector fields on $U$ which tends to $Z_0 \in \Gamma (H^{1,0}_0)$ as $k\rightarrow \infty$, where $H^{1,0}_k$ represents the $(1,0)$-bundle with respect to $J_k$. Since $f_k = \frac{1}{n} \,\log v_k$, we have 
$$Z_k f_k = \frac{Z_k v_k}{nv_k}$$ 
for every $k$. Since $L_{\theta_k} v_k =0$ on $B_{\theta_k}(x_k, 2r)$, the subellitic estimates Theorem \ref{t;subelliptic} implies that $Z_k f_k$ is uniformly bounded on $B_{\theta_k}(x_k, r)$ for every $k$. So is $\ov Z_k f_k$, and if $W_k$ is another sequence of vector fields, then $Z_k W_k f_k$ and $Z_k \ov W_k f_k$ are all uniformly bounded on $B_{\theta_k}(x_k, r)$ as $k\rightarrow \infty$. Therefore, if we denote by $\mathbf{R}_k$ and $\wt{\mathbf{R}}_k$ the Webster curvature tensors for $\theta_k$ and $\tilde\theta_k$ respectively, then \eqref{e;trans_curv} implies that 
$$|\wt{\mathbf{R}}_k|^2_{\tilde\theta_k}(F_k (x)) \leq C\la_k^{-2} \left\{|\mathbf{R}_k|^2_{\theta_k}(x) + A_k |\mathbf{R}_k|_{\theta_k}(x) + B_k \right\}$$ 
for every $x\in B_{\theta_k}(x_k, r)$, where $A_k$ and $B_k$ are some functions of the first and second covariant derivatives of $f_k$ with respect to the pseudohermitian structure $\theta_k$, which are uniformly bounded on $B_{\theta_k}(x_k, r)$ as $k\rightarrow \infty$. Since $\la_k \rightarrow \infty$ and $|\mathbf{R}_k|_{\theta_k}$ is uniformly bounded on $B_{\theta_k}(x_k, r)$ for every $k$, it turns out that $|\wt{\mathbf{R}}_k|_{\tilde\theta_k}\rightarrow 0$ uniformly on every compact subset of $M\setminus\{y_0\}$ by Lemma \ref{l;y_0}. Therefore, we see that the pseudohermitian manifold $(M\setminus\{y_0\}, \tilde\theta_0)$ has trivial Webster curvature. Similar argument with \eqref{e;trans_tor} implies that the torsion tensor of $\tilde\theta_0$ is also trivial. Therefore, we can conclude that $(M\setminus\{y_0\},\tilde\theta_0)$ is equivalent to the standard pseudohermitian structure of the Heisenberg group and therefore, $(M, J_0)$ is CR equivalent to the sphere by the removable singularity theorem. This contradicts to the hypothesis that $\Aut_{CR}(M, J_0)$ is compact and hence yields the conclusion of Proposition \ref{p;compact}

\end{document}